\documentclass[a4paper]{article}
\usepackage{amssymb}
\usepackage{amsmath}
\usepackage{amsthm}
\usepackage{bm}
\usepackage{graphicx}
\usepackage{ascmac}
\usepackage{comment}
\usepackage{mathtools}
\usepackage{mathrsfs}
\usepackage{tikz}
\usetikzlibrary{positioning}
\usetikzlibrary{cd}
\usepackage{pgfplots}
\usepackage[all]{xy}
\usepackage{latexsym}
\usepackage{amsfonts}
\usepackage{hyperref}
\usepackage{cleveref}
\usepackage[normalem]{ulem}
\usepackage[margin=25truemm]{geometry}
\theoremstyle{definition}
\newtheorem{defn}{Definition}[section]

\newtheorem{example}[defn]{Example}
\newtheorem{setting}[defn]{Setting}

\theoremstyle{plain}
\newtheorem{thm}[defn]{Theorem}
\newtheorem{prop}[defn]{Proposition}
\newtheorem{lem}[defn]{Lemma}
\newtheorem{cor}[defn]{Corollary}

\newtheorem{conj}[defn]{Conjecture}

\theoremstyle{remark}
\newtheorem{rem}[defn]{Remark}

\newcommand{\perf}{{\mathrm{D}_{\mathrm{pf}}}}
\newcommand{\qcoh}{{\mathrm{D}_{\mathrm{qcoh}}}}

\newcommand{\Idem}{\mathrm{Idem}}

\newcommand{\Spec}{\mathop{\mathrm{Spec}}}

\newcommand{\Th}{\mathrm{Th}}

\newcommand{\tSpec}{\mathrm{Spec}_{\triangle}}
\newcommand{\mo}{\mathrm}
\newcommand{\ol}{\overline}

\newcommand{\Supp}{\mathrm{Supp}}

\newcommand{\sO}{\mathscr{O}}
\newcommand{\Hom}{\mathrm{Hom}}

\newcommand{\rd}{\mathbb{R}}

\newcommand{\sNat}{\mathcal{N}\hspace{-0.08cm}\mathit{at}}

\newcommand{\Id}{\mathrm{Id}}

\newcommand{\ang}[1]{{\left\langle{#1}\right\rangle}}

\newcommand{\FM}{\mathrm{FM}}
\newcommand{\pFM}{\mathrm{pFM}}
\newcommand{\Ser}{\mathrm{Ser}}

\title{Proper Fourier-Mukai partners of abelian varieties and points outside the Fourier-Mukai loci in Matsui spectra}
\author{Hisato Matsukawa\thanks{Department of Mathematics, Graduate School of Science, Hokkaido University\\Kita 10, Nishi 8, Kita-Ku, Sapporo, Hokkaido, 060-0810, Japan\\Email:matsukawa.hisato.f4@elms.hokudai.ac.jp,hmnr0211@gmail.com}}
\date{}

\begin{document}
\maketitle

\begin{abstract}
We prove that any proper Fourier-Mukai partner of an abelian variety is again an abelian variety, by analyzing the Matsui spectrum of the derived category. 
This result was previously obtained by Huybrechts and Nieper-Wisskirchen in the case of smooth projective varieties. 
Our proof, however, extends the result to proper schemes using entirely different techniques.
More generally, we show that any scheme of finite type that is derived equivalent to an open subscheme of an abelian variety is itself an open subscheme of an abelian variety. 

We also study the structure of the Matsui spectrum outside the Fourier-Mukai locus.
For certain proper schemes, we show that the set of points lying outside the Fourier-Mukai locus in the Matsui spectrum has cardinality at least equal to that of the base field. 
This suggests the existence of additional geometric structures, such as moduli spaces, beyond the derived-equivalent part. 
As an application, we provide new counterexamples to conjectures of Ito, which predicted that the Serre invariant locus coincides with the Fourier-Mukai locus. 
While counterexamples involving K3 surfaces of Picard number one were previously given by Hirano-Ouchi, our examples arising from simple abelian varieties of dimension greater than two are the first of their kind.
\end{abstract}

\section{Introduction}
In 2008, Huybrechts and Nieper-Wisskirchen proved that any smooth projective Fourier-Mukai partner of an abelian variety is itself an abelian variety \cite{HN}.
Fourier-Mukai partners of abelian varieties have been extensively studied by many researchers, including Mukai \cite{Mu} and Orlov \cite{Or}.
While smoothness and properness are preserved under derived equivalence, projectivity is not.

In this paper, we show that any of finite type scheme derived equivalent to an open subscheme of an abelian variety is itself an open subscheme of an abelian variety.
In particular, by taking the whole abelian variety, this extends the result of Huybrechts and Nieper-Wisskirchen to proper Fourier-Mukai partners.
Our proof is entirely different from theirs: it is based on the Matsui spectrum of the derived category of the abelian variety.
We also rely on the classification of point objects by de Jong and Olsson \cite{JO}, and on the computation of the Fourier-Mukai locus by Ito and Matsui \cite{IM}.
Our proof is completely different from theirs: we make use of the Matsui spectrum of the derived category of the abelian variety.
We also rely on the classification of point objects by de Jong and Olsson \cite{JO}, and are inspired by the computation of the Fourier-Mukai locus by Ito and Matsui \cite{IM}.

Recently, Matsui defined the locally ringed space \(\tSpec T\) associated to a triangulated category \(T\), called the Matsui spectrum, in \cite{Ma1, Ma2}.
For a scheme \(X\), the Matsui spectrum \(\tSpec \perf(X)\) of its derived category of perfect complexes contains rich information about \(\perf(X)\), such as its derived autoequivalences.
However, the entire topological space is difficult to compute; the only known examples are the projective line and elliptic curves.

If \(Y\) is derived equivalent to \(X\), then there exists a natural open immersion \(Y \to \tSpec \perf(X)\).
The union of these immersions forms the geometric locus of the Matsui spectrum.
When \(X\) is a smooth proper variety, this geometric locus coincides with the (proper) Fourier-Mukai locus, as defined by Ito \cite{It}.
Unlike the whole space, the Fourier-Mukai locus is relatively more accessible to explicit computation.
It has been studied in various cases by Ito \cite{It, It2}, Ito-Matsui \cite{IM}, and Hirano-Ouchi \cite{HO, HO2}.

Although the structure outside the Fourier-Mukai locus remains elusive, it is believed to encode deep information about \(\perf(X)\), such as the classification of thick subcategories, semiorthogonal decompositions, exceptional objects, and some moduli spaces that are not derived equivalent to \(X\).

We prove that, in many cases, the set of points outside the Fourier-Mukai locus has cardinality at least \(\max\{\#k, \aleph_0\}\).
These cases include smooth proper varieties of dimension at least two with ample or anti-ample canonical bundle, simple abelian varieties of dimension at least two, K3 surfaces with Picard number one, and toric varieties of dimension at least two.
This suggests the existence of nontrivial geometric structures beyond the Fourier-Mukai locus.

Ito conjectured that the (proper) Fourier-Mukai locus coincides with the Serre invariant locus, that is, the set of points invariant under the action of the Serre functor \cite{It}.
This conjecture fails for K3 surfaces with Picard number one, as shown by Hirano-Ouchi \cite{HO2}, but no other counterexamples had been known.
Our result provides new counterexamples: simple abelian varieties of dimension at least two.

\begin{thm}[Derived equivalent schemes of open subschemes of abelian varieties]\label{thm:main_open_subschemes_of_abelian_varieties}
Let \(A\) be an abelian variety over an algebraically closed field \(k\) of characteristic zero, and let \(U \subset A\) be an open subscheme.
Let \(X\) be a scheme of finite type over \(k\) such that \(\dim X \le \dim A\) and \(\perf(U) \cong \perf(X)\) as \(k\)-linear triangulated categories.  
Then \(X\) is an open subscheme of an abelian variety \(B\) such that \(\perf(A) \cong \perf(B)\).
\end{thm}

\begin{cor}[Proper Fourier--Mukai partners of abelian varieties]\label{cor:main_proper_FM_partner_of_abelian_varieties}
Let \(A\) be an abelian variety over an algebraically closed field \(k\) of characteristic zero, and let \(X\) be a scheme of finite type over \(k\) such that \(\perf(A) \cong \perf(X)\) as \(k\)-linear triangulated categories.  
Then \(X\) is an abelian variety.
\end{cor}

\begin{cor}[\ref{cor:proper_Fourier_Mukai_locus_of_abelian_varieties}]\label{cor:FM_locus_and_pFM_locus_of_abelian_variety}
The proper Fourier-Mukai locus and the Fourier-Mukai locus of \(\perf(A)\) coincide.
\end{cor}

\begin{thm}\label{thm:main_proper_case}
Let \(X\) be a proper scheme of dimension \(d\) over a field \(k\).
Then we have the following inequality:
\[
  \# \left(\tSpec\perf(X) \setminus X\right) \ge 
  \begin{cases}
    \aleph_0 & (d = 1), \\
    \max\{\#k, \aleph_0\} & (d \ge 2).
  \end{cases}
\]
\end{thm}

\begin{thm}[Outside the Fourier-Mukai locus]\label{thm:main_outside_the_Fourier_Mukai_loci}
Let \(X\) be a proper scheme of dimension \(d\) over a field \(k\).
Assume that one of the following conditions holds:
\begin{enumerate}
  \item \(X\) is Gorenstein and its canonical bundle is ample or anti-ample (or more generally, \(\otimes\)-ample; see \cite{It2}).
  \item The characteristic of \(k\) is zero, and \(X\) is a smooth projective toric variety.
  \item The characteristic of \(k\) is zero, \(X\) is an abelian variety, and \(X\) is not isogenous to any product of elliptic curves.
  \item The characteristic of \(k\) is zero, and \(X\) is a K3 surface of Picard number one.
\end{enumerate}
Then the following inequality holds:
\[
  \# \left(\tSpec\perf(X) \setminus \tSpec^{\FM}\perf(X)\right) \ge 
  \begin{cases}
    \aleph_0 & (d = 1), \\
    \max\{\#k, \aleph_0\} & (d \ge 2).
  \end{cases}
\]
\end{thm}

Condition (4) strengthens the result of \cite[Theorem 5.8]{HO2}.

If \(d \ge 2\), Theorem \ref{thm:main_outside_the_Fourier_Mukai_loci} suggests the existence of geometric structures lying outside the Fourier-Mukai locus.
For example, if \(X\) is a Severi-Brauer scheme over a smooth proper variety \(Y\) and \(X\) satisfies one of the conditions above, then \(Y\) appears outside the Serre invariant locus as an open subscheme.

If \(d = 0\), there are no points outside \(X\), since \(X\) is affine.
If \(d = 1\), the lower bound \(\aleph_0\) is indeed attained when \(X = \mathbb{P}^1\), as shown in \cite[Example 4.10]{Ma1}.
Hence, \(\aleph_0\) gives the optimal lower bound in this case.
However, equality does not hold in general; for instance, see the case of elliptic curves in \cite{HO}.

We expect that the cardinality on the left-hand side satisfies the inequality
\[
\#\tSpec\perf(X)  \le \max\{\#k, \aleph_0\}
\]
for any scheme \(X\) of finite type.

Ito conjectured a relationship between the (proper) Fourier-Mukai locus and the Serre invariant locus.

\begin{conj}[Ito, {\cite[Conjecture 6.17]{It}}]\label{conj:Ito_proper_case}
Let \(X\) be a smooth projective variety over an algebraically closed field \(k\) of characteristic zero.  
Then the following equality holds:
\[
\tSpec^{\pFM}\perf(X) = \tSpec^{\Ser}\perf(X).
\]
\end{conj}

\begin{conj}[Ito, {\cite[Conjecture 6.14]{It}}]\label{conj:Ito_projective_case}
Under the same assumptions as in Conjecture \ref{conj:Ito_proper_case}, the following equality holds:
\[
\tSpec^{\FM}\perf(X) = \tSpec^{\Ser}\perf(X).
\]
\end{conj}

Conjectures \ref{conj:Ito_proper_case} and \ref{conj:Ito_projective_case} hold for smooth projective curves and for proper Gorenstein varieties whose canonical bundle is ample or anti-ample.
Further generalizations can be found in \cite{It2}.  
On the other hand, Hirano-Ouchi proved that K3 surfaces of Picard number one provide counterexamples to both Conjectures \ref{conj:Ito_proper_case} and \ref{conj:Ito_projective_case}.
Our result yields new counterexamples.

\begin{cor}
Conjectures \ref{conj:Ito_proper_case} and \ref{conj:Ito_projective_case} fail for abelian varieties satisfying the assumptions of Theorem \ref{thm:main_outside_the_Fourier_Mukai_loci} (3).
\end{cor}

\subsection*{Notations and assumptions}
All schemes are assumed to be topologically noetherian.

For a triangulated category \(T\) and a subset \(W \subset T\), we write \(\ang{W}\) for the smallest thick subcategory of \(T\) containing \(W\).  
We say that \(\ang{W}\) is classically generated by \(W\).

For a scheme \(X\), \(\perf(X)\) denotes the derived category of perfect complexes on \(X\), and \(\qcoh(X)\) denotes the derived category of \(\sO_X\)-modules with quasi-coherent cohomology.  
If \(X\) is regular, then \(\perf(X)\) is equivalent to the bounded derived category of coherent sheaves on \(X\).
If \(X\) is separated, then \(\qcoh(X)\) is equivalent to the derived category of quasi-coherent sheaves on \(X\).
If \(X\) is a scheme over a field \(k\), such triangulated categories admit a unique \(k\)-linear dg-enhancement (\cite{CNS}), so we may identify them with \(k\)-linear dg-enhanced triangulated categories when necessary.

\section{Preliminaries}
In this section, we introduce the Matsui spectrum of a triangulated category, as defined in \cite{Ma1}, \cite{Ma2}.  
Our definition of the structure sheaf is slightly different from the original one.  
It is a dg-analogue of the definition given in \cite{Mk}, which allows us to consider non-reduced structure sheaves.  
On the other hand, if \(T = \perf(X)\) for a smooth scheme \(X\), our structure sheaf coincides with Matsui's original definition on the proper Fourier-Mukai locus.

Let \(T\) be a triangulated category equipped with a dg-enhancement.  
We write \(\Th(T)\) for the set of thick subcategories of \(T\).

\begin{defn}[\cite{Ma1}]
  A thick subcategory \(P \subset T\) is called prime if the set
  \[\{ J \in \Th(T) \mid J \supsetneq P \}\]
  has a smallest element, denoted by \(\overline{P}\).
\end{defn}

The set of prime thick subcategories is denoted by \(\tSpec T\).

For a subset \(W \subset T\), its support is defined by
\[\Supp(W) = \{ P \in \tSpec T \mid W \not\subset P \}.\]
The set \(\tSpec T\) is regarded as a topological space with the topology whose family of closed subsets is generated by
\[\{ \Supp(t) \mid t \in T \}.\]

For a subset \(W \subset \tSpec T\), define
\[I_W = \{ a \in T \mid \Supp(a) \cap W = \emptyset \}.\]
This is a thick subcategory of \(T\).  
The idempotent-completion of the Verdier quotient is denoted by
\[D_W = \Idem(D / I_W).\]

Let \(\sO\) be the sheaf on \(\tSpec T\) associated to the presheaf
\[U \mapsto \mo{H}^0 \sNat(\Id_{D_U}, \Id_{D_U}),\]
where \(\sNat\) is the mapping chain complex in the dg-category \(\rd \underline{\Hom}(D_U, D_U)\), see \cite{To} for the definition.  
It is a sheaf of commutative rings.

\begin{defn}[Matsui, \cite{Ma1}, \cite{Ma2}]
  The ringed space
  \[\tSpec T = (\tSpec T, \sO)\]
  is called the Matsui spectrum of \(T\).
\end{defn}
It is actually a locally ringed space.

\begin{thm}\label{thm:foundamel_results_for_Matsui_spectrum}
  Let \(T\) be a triangulated category with a dg-enhancement.
  The followings hold.
  \begin{enumerate}
    \item (\cite[Theorem 2.8]{Mk}) The map
    \[\Supp : \Th(T) \to \{ Z \subset \tSpec T:{\rm subset} \}\]
    is an order-preserving injection.
    \item (\cite[Proposition 2.9]{Ma1}) Let \(I \subset T\) be a thick subcategory.
    Then, there is a canonical immersion 
    \[\tSpec(T/I) \to \tSpec(T).\]
    It is an open immersion if \(I\) is classically finitely generated.
    \item (\cite[Proposition 2.11]{Ma1}) There is a canonical isomorphism 
    \[\tSpec \Idem(T) \cong \tSpec T.\]
    \item Let \(X\) be a scheme.
    Then, there is a canonical open immersion
    \[X \to \tSpec \perf(X).\]
    \item Let \(X, Y\) be schemes with a derived equivalence \(\perf(X) \cong \perf(Y)\).
    Then, there is a canonical open immersion
    \[Y \to \tSpec \perf(X).\]
  \end{enumerate}
\end{thm}
\begin{proof}
  Properties (4) and (5) follow from the proof of \cite[Theorem 3.5]{Mk} and To{\"e}n's work (\cite{To}).
\end{proof}

\begin{defn}[Fourier-Mukai loci, proper Fourier-Mukai loci, and Serre invariant loci, \cite{It}, \cite{HO}]
  Let \(T\) be a proper \(k\)-linear triangulated category.
  \begin{enumerate}
    \item The proper Fourier-Mukai locus of \(T\) is defined by
    \[\tSpec^{\pFM} T = \bigcup_{X,\ \perf(X) \cong T } X\subset\tSpec\perf(X),\]
    where the union runs over all proper \(k\)-schemes \(X\) and equivalences \(\perf(X) \cong T\).
    Such information defines an open immersion \(X \to \tSpec T\) by Theorem \ref{thm:foundamel_results_for_Matsui_spectrum} (5).
    \item The Fourier-Mukai locus \(\tSpec^{\FM} T\) is defined similarly, but the union runs over all projective \(k\)-schemes \(X\) and equivalences \(\perf(X)\cong T\).
    \item If \(T\) admits a Serre functor \(S_T\), it naturally acts on \(\tSpec T\).
    The subspace of fixed points is called the Serre invariant locus and is denoted by \(\tSpec^{\Ser} T\).
  \end{enumerate}
\end{defn}

If there is a smooth proper scheme \(X\) such that \(T \cong \perf(X)\), there are inclusions
\[
  \tSpec^\FM T \subset \tSpec^\pFM T \subset \tSpec^\Ser T \subset \tSpec T.
\]
In this case, \(\tSpec^\FM T\) and \(\tSpec^\pFM T\) are smooth open subschemes of \(\tSpec T\).

\begin{example}
  \begin{enumerate}
    \item (Proposition \ref{prop:semiorthogonal_decomposition}) Let \(T\) be a classically finitely generated triangulated category (such as \(\perf(X)\); see \cite[Corollary 3.1.2]{BB}) with a semiorthogonal decomposition \((T_1, T_2)\).
    Then, there is an open immersion
    \[\tSpec T_1 \coprod \tSpec T_2 \to \tSpec T.\]

    \item (Corollary \ref{cor:exceptional_object}) Let \(T\) be a finitely generated triangulated category and \(E \in T\) an exceptional object.
    Then, \(E\) defines an open and closed point in \(\tSpec T\).

    \item (\cite{Ma2}) 
    \begin{eqnarray*}
      \tSpec \perf(\mathbb{P}^1_k) &=& \mathbb{P}^1_k \coprod \left(\coprod_{\mathbb{Z}} \Spec k\right),\\ \quad
      \tSpec^{\FM} \perf(\mathbb{P}^1_k) &=& \tSpec^{\pFM} \perf(\mathbb{P}^1_k) = \tSpec^{\Ser} \perf(\mathbb{P}^1_k) = \mathbb{P}^1_k.
    \end{eqnarray*}
    In this case, the Fourier-Mukai locus, proper Fourier-Mukai locus, and Serre invariant locus coincide.
    The points outside the Fourier-Mukai locus correspond to the exceptional objects \(\mathcal{O}(n)\) for \(n \in \mathbb{Z}\).
    Those \(\coprod_{\mathbb{Z}} \Spec k\) can be thought of as the Picard group \(\mathbb{Z}\) of \(\mathbb{P}^1\).

    \item (\cite{HO}, \cite{Ma2}) Let \(E\) be an elliptic curve over an algebraically closed field \(k\) of characteristic \(0\).
    Then we have
    \begin{eqnarray*}
      \tSpec \perf(E) &=& \tSpec^{\FM} \perf(E)  \tSpec^{\pFM} \perf(E)\\ &=& \tSpec^{\Ser} \perf(E) = E \coprod \left( \coprod_{(r,d) \in I} M_{r,d} \right),
    \end{eqnarray*}
    where \(
      I = \{ (r,d) \in \mathbb{Z}^2 \mid r \geq 1, \gcd(r,d) = 1 \}\).
    For each \((r,d) \in I\), the corresponding \(M_{r,d}\) is the moduli space of stable vector bundles of rank \(r\) and degree \(d\), which is (non-canonically) isomorphic to \(E\).

    \item (\cite{Mk}) Let \(X\) be a Severi-Brauer scheme of relative dimension \(r\) over a noetherian scheme \(Y\).
    Then, there is an open immersion
    \[
      X \coprod \left( \coprod_{\mathbb{Z}} Y \right) \to \tSpec \perf(X).
    \]
    For each \(n \in \mathbb{Z}\), the \(n\)-th copy of \(Y\) corresponds to the semiorthogonal decomposition
    \[\left(\perf(X,Y)_n, \ang{ \bigcup_{m=n+1}^{n+r} \perf(X,Y)_m }\right)\] described in \cite{Be}.
    Those \(\coprod_{\mathbb{Z}} Y\) can also be thought of as the relative Picard group over \(Y\), and they are not contained in the Serre invariant locus.
  \end{enumerate}
\end{example}

\section{Semiorthogonal decompositions, proper objects, prime objects, and exceptional objects}

\begin{prop}\label{prop:semiorthogonal_decomposition}
  Let \(D\) be a classically finitely generated triangulated category, and \((D_1, D_2)\) a semiorthogonal decomposition of \(D\).
  Then, there is a canonical open immersion
  \[\tSpec D_1 \coprod \tSpec D_2 \to \tSpec D.\]
\end{prop}
\begin{proof}
Let \(L : D \to D_1\) be the composition of the localization functor \(D \to D/D_2\) and its right adjoint \(D/D_2 \cong D_1\).
The functor \(L\) induces an immersion \(\iota_1 : \tSpec D_1 \to \tSpec D\) by \cite{Ma1}.

Let \(F \in D\) be a classical generator of \(D\), and consider an exact triangle \(G \to F \to L(F)\) in \(D\).
By definition, \(G \in D_2\), and \(\ang{G} = D_2\).
This implies that \(\iota_1\) is an open immersion.
Dually, there is an open immersion \(\iota_2 : \tSpec D_2 \to \tSpec D\).

We need to show that the images of \(\iota_1\) and \(\iota_2\) are disjoint.
Assume there exists a point \(x \in \mathrm{Im}(\iota_1) \cap \mathrm{Im}(\iota_2)\).
Let \(P_x\) be the corresponding prime thick subcategory.
By definition, \(P_x\) contains both \(D_1\) and \(D_2\), thus it contains \(D = \ang{D_1, D_2}\).
Since \(P_x\) is a prime thick subcategory, \(P_x \neq D\) must hold.
This is a contradiction.
\end{proof}

\begin{defn}
Let \(k\) be a field, and \(D\) a \(k\)-linear triangulated category.
We say an object \(a \in D\) is left proper if for any \(b \in D\), \(\dim_k \Hom^*(b,a) < \infty\).
Dually, \(a\) is right proper if it is left proper in the opposite category \(D^{op}\).
If \(a\) is both left proper and right proper, we say \(a\) is proper.
\end{defn}
If all objects of \(D\) are proper, then \(D\) is said to be proper.

\begin{prop}\label{prop:object_properness}
Let \(X\) be a separated scheme of finite type over \(k\), and let \(F \in \perf(X)\).
Then the following conditions are equivalent:
\begin{enumerate}
  \item \(F\) is left proper.
  \item \(F\) is right proper.
  \item \(\Supp_X(F)\) is a proper \(k\)-scheme.
\end{enumerate}
\end{prop}

\begin{proof}
It is well known that (3) implies (1) and (2); see, for example, the proof of \cite[Proposition 3.26]{Or3}.
We show that (1) implies (3).
Assume, for contradiction, that \(\Supp_X(F)\) is not proper.
By the proof of \cite[Proposition 3.30]{Or3}, there exists an affine closed curve \(C \subset \Supp_X(F)\).
By \cite[Lemma 3.4]{Th}, we may take \(G,H \in \perf(X)\) such that \(\Supp_X(G) = \Supp_X(H) = C\).
By adjunction, we may replace \(F\) by \(\mathbf{R}\underline{\Hom}(G,F)\) and assume \(\Supp_X(F) = C\).
There is a spectral sequence 
\[
E_2^{p,q} = \bigoplus_{i-j=q} \Hom^p(\mathcal{H}^j(H), \mathcal{H}^i(F)) \Rightarrow \Hom^{p+q}(H,F).
\]
Let \(i_0, j_0\) be the minimal and maximal integers such that \(\Supp_X(\mathcal{H}^i(F)) = C\) and \(\Supp_X(\mathcal{H}^j(H)) = C\), respectively.
By Lemma \ref{lem:finiteness_of_Hom}, \(\dim_k E_2^{p,q} < \infty\) if \(q < i_0 - j_0\), equals zero if \(p < 0\) or \(q \ll 0\), and \(\dim_k E_2^{0,i_0 - j_0} = \infty\).
This implies \(\dim_k \Hom^{i_0 - j_0}(H,F) = \infty\), a contradiction.

To prove that (2) implies (3), we apply a similar argument by replacing \(F\) with \(F \otimes G\).
The rest of the proof proceeds analogously.
\end{proof}

\begin{lem}\label{lem:finiteness_of_Hom}
  Let \(X = \Spec R\) be an irreducible affine scheme of dimension 1 of finite type over \(k\), and let \(M, N\) be finitely generated \(R\)-modules.
  Then, \(\dim_k \Hom^0(M, N) = \infty\) if \(\Supp_X(M) = \Supp_X(N) = X\).
  Otherwise, \(\dim_k \mo{Ext}^i(M, N) < \infty\) for all \(i\).
\end{lem}

\begin{proof}
  If \(\Supp_X(M) \neq X\) or \(\Supp_X(N) \neq X\), then each \(\mo{Ext}^i(M, N)\) is a finitely generated \(R\)-module whose support has dimension 0.
  Thus, these modules are Artinian and have finite \(k\)-dimension.

  Now assume \(\Supp_X(M) = \Supp_X(N) = X\).
  After replacing \(M, N\) with \(fM, fN\) for a non-zero-divisor \(f \in R\), we may assume that for any non-zero-divisor \(g \in R\), \(gx \neq 0\) for any nonzero \(x \in M\) and \(x \in N\).
  Let \(I \subset R\) be the nilradical.
  Take the minimal \(n\) such that \(I^n N = 0\).
  Replace \(M\) by \(M / I M\), \(N\) by \(I^{n-1} N\), and \(R\) by \(R / I\).
  Then, \(R\) is an integral domain and \(M, N\) are torsion-free \(R\)-modules.
  Take a nonzero homomorphism \(\phi: M \otimes \mathrm{Frac}(R) \to N \otimes \mathrm{Frac}(R)\).
  Choose \(f \in R \setminus \{0\}\) such that \(\phi(f M) \subset N\).
  This induces an injection \(R \to \Hom(M, N),\  g \mapsto f g \phi\), which implies \(\dim_k \Hom(M, N) = \infty\).
\end{proof}

By Proposition \ref{prop:object_properness}, we may recover the well-known fact that \(X\) is proper if and only if \(\perf(X)\) is proper \cite[Proposition 3.30]{Or3}.

From now on, we assume the following setting.

\begin{setting}
  Let \(D\) be a triangulated category, and \(\overline{D}\) a triangulated category admitting all small direct sums and satisfying Brown representability. Suppose that \(D\) is a full triangulated subcategory of \(\overline{D}\).
\end{setting}

For example, if \(D\) is an enhanced triangulated category, then its derived category of dg-modules satisfies this condition.
If \(X\) is a quasi-compact separated scheme, then \(\perf(X)\) is a full subcategory of \(\qcoh(X)\), and this embedding satisfies the condition.

In this setting, for any thick subcategory \(I \subset D\), the localization functor \(D \to \Idem(D/I)\) has a fully faithful right adjoint \(\Idem(D/I) \to \ol{D}\) by Brown's representability theorem.
We will identify \(\Idem(D/I)\) with a full subcategory of \(\overline{D}\) via the adjoint.

\begin{defn}
  In the above setting, let \(p \in \overline{D}\).
  We say \(p\) is a prime object if the following conditions are satisfied.
  \begin{enumerate}
    \item The thick subcategory \(P_p := \{ F \in D \mid \Hom^*(F, p) = 0 \}\) is a prime thick subcategory of \(D\).
    \item \(p\) is contained in \(\Idem(\overline{P_p}/P_p)\).
  \end{enumerate}
  The thick subcategory \(P_p\) is called the prime thick subcategory defined by the prime object \(p\).
\end{defn}

\begin{example}
  \begin{enumerate}
    \item Let \(X\) be a scheme, \(x \in X\) a point, and \(F\) a perfect complex with support \(\overline{\{x\}}\).
    Then, \(F \otimes \sO_{X,x} \in \qcoh(X)\) is a prime object defining the point \(x\).
    Especially, if \(x\) is a closed point, the point \(x\) is defined by the perfect complex \(F\).
    If \(x\) is a regular point, the residue field \(k(x)\) is a prime object defining \(x\).

    \item Let \(E \in D\) be an exceptional object.
    Then, \(E\) is a prime object by \cite{HO}.
  \end{enumerate}
\end{example}

\begin{prop}
  \begin{enumerate}
    \item Let \(P\) be a prime thick subcategory and \(p \in \Idem(\ol{P}/P)\) a non-zero object.
    Then \(p\) is a prime object defining \(P\).
    
    \item For any prime thick subcategory \(P\), there exists a prime object defining \(P\).
    
    \item Let \(p \in \overline{D}\) be a prime object, and let \(q \in \overline{D}\).
    Then \(q\) is a prime object defining \(P_p\) if and only if \(q \in \ang{p}\) and \(q \neq 0\).
    
    \item Let \(P\) be a prime thick subcategory, and \(p, q\in\ol{D}\) prime objects defining \(P\).
    Then \(p \in \Idem(D)\) if and only if \(q \in \Idem(D)\).
  \end{enumerate}
\end{prop}

\begin{proof}
  The assertions (2), (3), and (4) follow immediately from (1).
  Let us prove (1).
  Take any object \(F \in D\), and let \(L: D \to D/P\) be the localization functor.
  Then, by adjunction, we have \(\Hom^*(F, p) = \Hom^*(L(F), p)\).
  The following lemma shows that \(\Hom^*(L(F), p) = 0\) if and only if \(L(F) = 0\), which is equivalent to \(F \in P\).
\end{proof}

\begin{lem}\label{lem:non_zeroness_of_prime_object}
  Let \(D\) be a local triangulated category, i.e., the trivial thick subcategory \(0\) is a prime thick subcategory.
  Then for any object \(F \in D \setminus 0\) and \(p \in \overline{0} \setminus 0\), we have
  \[
    \Hom^*(F, p) \neq 0,\quad \Hom^*(p, F) \neq 0.
  \]
\end{lem}

\begin{proof}
  Since \(\overline{0}\) is the smallest non-zero thick subcategory, we have \(p \in \ang{F}\).
  If \(\Hom^*(F, p) = 0\), then \(\Hom^*(E, p) = 0\) for any \(E \in \ang{F}\), and in particular \(\Id_p = 0\), which leads a contradiction.
  By duality, we also get \(\Hom^*(p, F) \neq 0\).
\end{proof}

\begin{prop}\label{prop:small_prime_object_defines_closed_point}
  Let \(p \in \Idem(D)\) be a prime object.
  Then \(P_p\) is a minimal prime thick subcategory.
  In particular, the corresponding point \(x_p \in \tSpec D\) is a closed point.
\end{prop}

\begin{proof}
  The last assertion follows from the previous proposition.
  Assume there exists a prime thick subcategory \(Q\) strictly contained in \(P_p\), and let \(q\) be a prime object defining \(Q\).
  After replacing \(D\) with \(\Idem(D/Q)\), we may assume \(Q = 0\).
  Then, by Lemma~\ref{lem:non_zeroness_of_prime_object}, we have \(\Hom^*(q, p) \neq 0\).
  On the other hand, since \(q \in \overline{Q} \subset P_p\), we have \(\Hom^*(q, p) = 0\), which leads a contradiction.
\end{proof}

\begin{cor}\label{cor:exceptional_object}
  Let \(E \in D\) be an exceptional object.
  Then the associated prime thick subcategory \(P_E\) defines a closed point in \(\tSpec D\).
  Moreover, if \(D\) is classically finitely generated, the point is open and closed.
\end{cor}

\begin{cor}
  Let \(X\) be a scheme and \(x \in X\) a closed point.
  Then the corresponding point \(x \in \tSpec(\perf(X))\) is a closed point.
\end{cor}

\section{Proof of Theorem \ref{thm:main_open_subschemes_of_abelian_varieties}}

Let \(A\) be an abelian variety over an algebraically closed field \(k\) of characteristic zero. Let \(B \subset A\) be a abelian subvariety, and \(H \in \operatorname{NS}(B)_{\mathbb{Q}}\) a rational class in the N\'eron-Severi group of \(B\).

A coherent sheaf \(F\) on \(A\) is called semi-homogeneous of type \((B, H)\) if there exists a closed point \(a \in A\) and a semi-homogeneous vector bundle \(\widetilde{F}\) on \(B\) of slope \(H\) such that \(F \cong T_a^* i_* \widetilde{F}\), where \(i: B \to A\) is the natural inclusion and \(T_a: A \to A\) is the translation map defined by \(a' \mapsto a + a'\).

By \cite{Mu}, there exists a fine moduli space \(M_{B,H}\) parameterizing simple semi-homogeneous sheaves of type \((B, H)\), along with a universal sheaf \(E_{B,H}\) on \(M_{B,H} \times A\). By construction, \(M_{B,H}\) is isomorphic to an abelian variety (though not canonically pointed).
The Fourier-Mukai transform
\[
\Phi_{B,H} := \Phi_{E_{B,H}}: \perf(M_{B,H}) \to \perf(A)
\]
induced by the kernel \(E_{B,H}\) is an equivalence of triangulated categories, as shown in \cite[Proposition 4.11]{Mu}. Furthermore, it induces an open immersion
\[
\Phi_{B,H}: M_{B,H} \to \tSpec \perf(A).
\]

\begin{prop}
  There is a canonical open immersion 
  \[
    \coprod \Phi_{B,H}: \coprod_{B \subset A,\ H \in \operatorname{NS}(B)_{\mathbb{Q}}} M_{B,H} \to \tSpec \perf(A).
  \]
\end{prop}

\begin{proof}
  Let \(b \in M_{B,H}\) be a closed point, and let \(E_b\) be the corresponding semi-homogeneous sheaf. Denote by \(P_b \subset \perf(M_{B,H})\) the corresponding prime thick subcategory. Since \(k(b)\) is a prime object defining \(P_b\), its image \(E_b = \Phi_{B,H}(k(b))\) is the prime object defining the prime thick subcategory \(\Phi_{B,H}(P_b) \subset \perf(A)\).
  By Lemma \ref{lem:uniqueness_of_point_object}, two non-isomorphic semi-homogeneous sheaves define different prime thick subcategories. Since each \(\Phi_{B,H}\) is an open immersion, if 
  \(\Phi_{B,H}(M_{B,H}) \cap \Phi_{B',H'}(M_{B',H'}) \neq \emptyset\),
  then there exist closed points \(b \in M_{B,H}\) and \(b' \in M_{B',H'}\) such that \(\Phi_{B,H}(b) = \Phi_{B',H'}(b')\), i.e., \(E_b \cong E_{b'}\) up to shift. But then \(B = B'\), \(H = H'\), and \(b = b'\). Therefore, the images are disjoint.
\end{proof}

\begin{lem}\label{lem:uniqueness_of_point_object}
  Let \(X\) be a scheme, \(x \in X\) a regular point, \(P_x\) the prime thick subcategory of \(\perf(X)\) corresponding to \(x\), and \(F \in \qcoh(X)\) a prime object defining \(P_x\).
  Assume
  \[
    \Hom^i(F, F) =
    \begin{cases}
      0 & (i < 0), \\
      k & (i = 0),
    \end{cases}
  \]
  where \(k\) is a field. Then \(F\) is isomorphic to a shift of \(k(x)\).
\end{lem}
\begin{proof}
  If necessary, by localizing \(X\), we may assume that \(x\) is a closed point and that \(F\in\perf(X)\).
  Then, the assertion follows from the proof of \cite[Proposition 2.2]{BO}.
\end{proof}

\begin{proof}[Proof of Theorem \ref{thm:main_open_subschemes_of_abelian_varieties}]
Since \(\perf(X)\) is smooth, we have \(\Delta_* \sO_X \in \perf(X \times X)\), where \(\Delta\) denotes the diagonal morphism.
Hence, \(\Delta(X) = \Supp_{X \times X}(\Delta_* \sO_X)\) is closed in \(X \times X\), and thus \(X\) is separated.
The smoothness of \(\perf(X)\) also implies that \(X\) is smooth by \cite[Proposition 3.13]{Lu}, (see also the proof of \cite[Proposition 3.31]{Or3}).
Moreover, \(X\) is connected, because \(\Gamma(X, \sO_X) \cong \mathrm{HH}^0(X, \sO_X) \cong \mathrm{HH}^0(U, \sO_U) \cong \Gamma(U, \sO_U)\) has no nontrivial idempotents.

There are natural open immersions \(X \to \tSpec \perf(U)\) and \(\tSpec \perf(U) \to \tSpec \perf(A)\) by Theorem \ref{thm:foundamel_results_for_Matsui_spectrum} (5), (2), and (3), together with the natural isomorphism \(\perf(U) \cong \Idem(\perf(A)/\perf_{A \setminus U}(A))\).
Let \(x \in X\) be a closed point, and let \(P_x\) be the corresponding prime thick subcategory of \(\perf(U)\), and \(\overline{P_x}\) the corresponding prime thick subcategory of \(\perf(A)\).
Let \(\phi: \perf(X) \xrightarrow{\sim} \perf(U)\) be the equivalence given by hypothesis, and set \(F = \phi(k(x))\).
Let \(i: U \to A\) denote the natural inclusion, and define \(\overline{F} = \mathbf{R} i_*(F) \in \qcoh(A)\).
The object \(F\) is a prime object defining \(P_x\), and hence \(\overline{F}\) is a prime object defining \(\overline{P_x}\).
Since \(k(x)\) is proper, \(F\) is also proper.
By Proposition \ref{prop:object_properness}, \(\Supp_X(F)\) is a proper \(k\)-scheme.
Thus, \(\overline{F}\) is a perfect complex, and \(\mo{Ext}^i(\overline{F}, \overline{F}) \cong \mo{Ext}^i(F, F) \cong \mo{Ext}^i(k(x), k(x))\) for all \(i\).
By the assumption \(\dim X \le \dim A\), \(\overline{F}\) is a point object in the sense of \cite{JO}, and is isomorphic (up to shift) to a simple semi-homogeneous sheaf \(G\) by \cite{JO}.
This implies that the image of \(x\) in \(\tSpec \perf(A)\) is the point corresponding to \(G\) in \(\coprod_{B, H} M_{B,H}\).
Since \(\coprod_{B, H} M_{B,H}\) is open in \(\tSpec \perf(A)\), it is generalization-closed, and hence the image of \(X\) lies in \(\coprod_{B, H} M_{B,H}\).
As \(X\) is connected, it must be contained in a connected component \(M_{B,H}\), and each \(M_{B,H}\) is isomorphic to an abelian variety and is a Fourier-Mukai partner of \(A\).
\end{proof}

\begin{proof}[Proof of Corollary \ref{cor:main_proper_FM_partner_of_abelian_varieties}]
By the same reasoning as in the proof of Theorem \ref{thm:main_open_subschemes_of_abelian_varieties}, \(X\) is separated.
By \cite[Proposition 3.31]{Or3}, \(X\) is smooth and proper.
Moreover, by \cite[Theorem 2.3]{Ka}, we have \(\dim X = \dim A\).
By Theorem \ref{thm:main_open_subschemes_of_abelian_varieties}, \(X\) is isomorphic to one of the \(M_{B,H}\), and hence is an abelian variety.
\end{proof}

\begin{cor}\label{cor:proper_Fourier_Mukai_locus_of_abelian_varieties}
  \[
    \tSpec^{\pFM} \perf(A) = \tSpec^{\FM} \perf(A) = \coprod_{B \subset A,\ H \in \operatorname{NS}(B)_{\mathbb{Q}}} M_{B,H}.
  \]
\end{cor}
Corollary \ref{cor:proper_Fourier_Mukai_locus_of_abelian_varieties} is a generalization of \cite[Theorem 4.11]{HO} and a strengthening of \cite[Theorem 5.6, Corollary 5.7]{IM}.

In the proof of Theorem \ref{thm:main_open_subschemes_of_abelian_varieties}, the separatedness of the proper Fourier-Mukai locus plays a central role.
Unfortunately, the proper Fourier-Mukai locus is not separated in general; see \cite{It}.
In such cases, we expect that the maximal commutative subspaces (see \cite[Theorem B.9]{Mk}) of the proper Fourier-Mukai locus serve as suitable candidates for proper Fourier-Mukai partners.

\section{Generation of derived categories}

\begin{lem}\label{lem:generation_for_curve}
  Let \(C\) be a geometrically integral smooth proper curve over a field \(k\).
  Then, \(\ang{\sO_C} \neq \perf(C)\).  
  In particular, for any object \(F \in \ang{\sO_C}\), the support \(\Supp_C(F)\) is either \(C\) or \(\emptyset\).
\end{lem}

\begin{proof}
  By Serre duality, we have \(\mo{Ext}^i(\sO_C, \sO_C) = 0\) for \(i \ne 0,1\) (\cite[Proposition 3.13]{Hu}).
  By \cite[Proposition 2.13]{HKO}, the thick subcategory \(\ang{\sO_C}\) has no nontrivial thick subcategories.
  On the other hand, \(\perf(C)\) does have nontrivial thick subcategories; for example, for a closed point \(c \in C\), the subcategory \(\ang{k(c)}\) is a nontrivial thick subcategory of \(\perf(C)\).  
  This shows that \(\ang{\sO_C} \ne \perf(C)\).

  Suppose, for contradiction, that there exists an object \(F \in \ang{\sO_C}\) such that \(\Supp_C(F) \ne C\) or \(\emptyset\).
  Then \(\Supp_C(F)\) is a finite set of closed points; let \(Z = \Supp_C(F)\).  
  It follows that \(\ang{F} = \perf_Z(C) \subset \ang{\sO_C}.\)
  There is a one-to-one correspondence between thick subcategories of \(\perf(C)\) containing \(\perf_Z(C)\) and thick subcategories of 
  \(\Idem(\perf(C)/\perf_Z(C)) \cong \perf(C \setminus Z)\).
  Since \(C \setminus Z\) is affine, we have \(\ang{\sO_{C \setminus Z}} = \perf(C \setminus Z)\).  
  This implies \(\ang{\sO_C} = \perf(C)\), which contradicts the first part.
\end{proof}

\begin{prop}\label{prop:generation}
  Let \(X\) be a scheme (not necessarily over \(k\)), \(C\) a proper curve over \(k\), and \(f: C \to X\) a morphism.  
  Assume that the image \(f(C)\) is not a point.
  Then for any object \(F \in \ang{\sO_X}\), we have \(f^{-1}(\Supp_X(F)) = C\) or \(\emptyset\).
  In particular, \(\ang{\sO_X} \ne \perf(X)\).
\end{prop}

\begin{proof}
  The last assertion follows from \cite[Theorem 3.15]{Th}.
  By replacing \(k\) with its algebraic closure, base changing and normalizing \(C\), and taking a connected component, we may assume that \(C\) is a smooth proper integral curve over an algebraically closed field \(k\).
  We then have \(\ang{f^*(\ang{\sO_X})} = \ang{f^* \sO_X} = \ang{\sO_C}\).
  The assertion follows from Lemma \ref{lem:generation_for_curve}.
\end{proof}

\begin{cor}\label{cor:generation_for_line_bundles}
  Let \(L\) be a line bundle on \(X\).  
  Then the same results hold for \(\ang{L}\).
\end{cor}

\begin{proof}
  Apply the autoequivalence \(- \otimes L: \perf(X) \to \perf(X)\) to the previous results.
\end{proof}

\begin{example}
  \begin{enumerate}
    \item Let \(X\) be a proper scheme of dimension greater than \(1\) over a field \(k\).  
    Then, \(\ang{\sO_X} \ne \perf(X)\).

    \item Let \(\overline{X}\) be a proper scheme of dimension greater than \(2\) over a field \(k\), and let \(Z \subset \overline{X}\) be a closed subset of codimension greater than \(2\).  
    Set \(X = \overline{X} \setminus Z\).  
    Then, \(\ang{\sO_X} \ne \perf(X)\).
    In fact, such a subset \(Z\) can be chosen to be specialization-closed and of codimension \(\ge 2\).  
    Let us sketch the argument.
    Let \(d = \dim \overline{X}\), and take a point \(\eta \in \overline{X}\) of dimension \(d\).
    Choose a transcendence basis \(x_1, \ldots, x_d\) of the residue field \(k(\eta)\) over \(k\).
    Then, there exists a proper curve \(C\) over the intermediate field \(k(x_1, \ldots, x_{d-1})\), and an immersion \(C \to \overline{X}\) whose image avoids any point of codimension \(\ge 2\).  
    In particular, this curve lies entirely inside \(X = \overline{X} \setminus Z\), and Proposition \ref{prop:generation} shows that \(\ang{\sO_X} \ne \perf(X)\).
  \end{enumerate}
\end{example}

\subsection*{Conjectures}
We formulate several conjectures as natural generalizations of Proposition~\ref{prop:generation}.

\begin{conj}[Quasi-affineness]\label{conj:quasi-affineness}
Let \(X\) be a separated scheme.  
Then, \(X\) is quasi-affine if and only if \(\ang{\sO_X} = \perf(X)\).
\end{conj}

The “only if” direction of Conjecture~\ref{conj:quasi-affineness} holds.  
Conversely, if \(\ang{\sO_X} = \perf(X)\), then Proposition \ref{prop:generation} implies that there exists no non-constant morphism from a proper curve over any field \(k'\) to \(X\).
Note that the conjecture fails without the separatedness assumption; explicit counterexamples exist.

\begin{conj}\label{conj:properness_line_bundle}
Let \(X\) be a proper scheme of dimension \(d\) over a field \(k\) and \(L_1, \ldots, L_d\) line bundles on \(X\).  
Then, \(\ang{L_1,\ldots,L_d} \ne \perf(X)\).
\end{conj}

\begin{rem}\label{rem:properness_counter_example}
Conjecture~\ref{conj:properness_line_bundle} is false for projective schemes if we allow \(d+1\) line bundles.  
Indeed, if \(L\) is a very ample line bundle, then \(\ang{\sO_X, L, L^{\otimes 2}, \ldots, L^{\otimes d}} = \perf(X)\) by \cite[Theorem 4]{Or2}.
\end{rem}

By Proposition~\ref{prop:generation}, Conjecture~\ref{conj:properness_line_bundle} holds when \(\dim X = 1\), i.e., for proper curves.

\begin{conj}\label{conj:properness_objects}
Let \(X\) be a proper scheme of dimension \(d\) over a field \(k\), and let \(F_1,\ldots,F_d \in \perf(X)\) be objects satisfying
\[
\mathrm{Ext}^p(F_i, F_i) = 
\begin{cases}
0 & (p<0) \\
k & (p=0)
\end{cases}
\quad \text{for all } i.
\]
Then, \(\ang{F_1,\ldots,F_d} \ne \perf(X)\).
\end{conj}

For example, if \(d \ge 2\), \(L\) is a line bundle on \(X\), and \(x \in X\) is a regular closed point, then Proposition~\ref{prop:generation} implies that \(\ang{L, k(x)} \ne \perf(X)\).  
Conjecture~\ref{conj:properness_objects} holds for proper Gorenstein curves.

Conjecture~\ref{conj:properness_objects} implies Conjecture~\ref{conj:properness_line_bundle} for proper geometrically integral schemes.

\begin{conj}\label{conj:Krull_dimension}
Let \(X, Y\) be separated schemes of finite type over a field \(k\), and assume \(\perf(X) \cong \perf(Y)\).  
Then, \(\dim X = \dim Y\).
\end{conj}

Conjecture~\ref{conj:Krull_dimension} is known to hold for smooth projective varieties, as shown in \cite[Theorem 2.3]{Ka}.  
A similar proof extends to proper Gorenstein schemes, where one can define a Serre functor using the canonical bundle.

Conjecture~\ref{conj:properness_objects} implies Conjecture~\ref{conj:Krull_dimension} for projective schemes by Remark~\ref{rem:properness_counter_example}.

\section{Proof of Theorems \ref{thm:main_proper_case} and \ref{thm:main_outside_the_Fourier_Mukai_loci}}

\begin{lem}\label{lem:increasing_sequence}
Let \(X\) be a proper scheme over a field \(k\) of dimension \(\ge 2\) and \(X_0\) the set of closed points of \(X\).  
Let \(\Lambda\) be a well-ordered set of cardinality \(\max\{\#k,\aleph_0\}\), equipped with a bijection \(\Lambda \to X_0\), \(\lambda \mapsto x_\lambda\).  
For each \(\lambda \in \Lambda\), define \(Z_\lambda = \{x_{\lambda'} \in X_0 \mid \lambda' < \lambda\}\), \(I_Z = \ang{\sO_X, \perf_Z(X)}\), and \(I_\lambda = I_{Z_\lambda}\).
Then the following sequence is strictly increasing in \(\tSpec \perf(X)\):
\[
\Supp(I_0) \subsetneq \Supp(I_1) \subsetneq \cdots \subsetneq \Supp(I_\lambda) \subsetneq \Supp(I_{\lambda+1}) \subsetneq \cdots.
\]
\end{lem}

\begin{proof}
By Theorem~\ref{thm:foundamel_results_for_Matsui_spectrum}~(1), it suffices to prove that \(I_\lambda \subsetneq I_{\lambda+1}\) for all \(\lambda \in \Lambda\).
Fix \(\lambda \in \Lambda\), and choose \(F \in \perf(X)\) such that \(\Supp_X(F) = \{x_\lambda\}\).  
Then clearly \(F \in I_{\lambda+1}\).  
To show \(F \notin I_\lambda\), observe that \(I_\lambda\) is the filtered colimit of the form
\[
I_\lambda = \mathop{\mo{colim}}_{Z \subset Z_\lambda: \text{ finite}} I_Z,
\]
which is a set-theoretic union on objects.
It thus suffices to show \(F \notin I_Z\) for every finite \(Z \subset Z_\lambda\).
Since \(X\) is of dimension \(\ge 2\), there exists a closed curve \(C \subset X\) such that \(C\) contains \(x_\lambda\) and \(C \cap Z = \emptyset\).
Hence, by Proposition~\ref{prop:generation}, \(F \notin I_Z = \ang{\sO_X, \perf_Z(X)}\).
This proves \(F \notin I_\lambda\), and thus \(I_\lambda \subsetneq I_{\lambda+1}\).
\end{proof}

\begin{proof}[Proof of Theorem \ref{thm:main_proper_case}]
First, assume \(d \ge 2\).  
Since \(X \subset \Supp(\sO_X) = \Supp(I_0)\), the assertion immediately follows from Lemma~\ref{lem:increasing_sequence}.
Now assume \(d = 1\).  
In this case, \(X\) is a projective scheme over \(k\).  
There exists a very ample line bundle \(L\) on \(X\).  
By \cite[Theorem 4]{Or2}, for any two distinct integers \(n,m\), we have \(\ang{L^{\otimes n}, L^{\otimes m}} = \perf(X)\).
On the other hand, by Corollary~\ref{cor:generation_for_line_bundles}, we have \(\ang{L^{\otimes n}}, \ang{L^{\otimes m}} \ne \perf(X)\).  
Therefore, \(\ang{L^{\otimes n}} \ne \ang{L^{\otimes m}}\).  
Since \(X \subset \Supp(L^{\otimes n})\) for any \(n\), and by Theorem~\ref{thm:foundamel_results_for_Matsui_spectrum}~(1), this implies that there exist at least countably infinitely many closed subsets of \(\tSpec \perf(X)\) containing \(X\), which proves the assertion.
\end{proof}

\begin{proof}[Proof of Theorem \ref{thm:main_outside_the_Fourier_Mukai_loci} (1)]
This follows directly from \cite[Corollary 4.14]{It} and Theorem~\ref{thm:main_proper_case}.
\end{proof}

\begin{proof}[Alternative proof of Theorem \ref{thm:main_outside_the_Fourier_Mukai_loci} (1) in the case \(d = 1\)]
By Proposition~\ref{prop:generation}, we have \(X \subset \Supp(\sO_X)\), but \(\Supp(\sO_X) \subsetneq \tSpec \perf(X)\).  
This implies that there exists a point \(x \in \tSpec \perf(X) \setminus X\).  
By the proof of \cite[Proposition 5.3]{HO}, no power of the Serre functor \(S_X\) fixes \(x\).  
Therefore, the sequence \(x, S_X(x), S_X^2(x), \ldots\) consists of mutually distinct points in \(\tSpec \perf(X)\).
\end{proof}

\begin{proof}[Proof of Theorem \ref{thm:main_outside_the_Fourier_Mukai_loci} (2), (4)]
The proof is similar to that of Theorem \ref{thm:main_proper_case}.  
We need to show that the Fourier-Mukai locus is contained in the support of any line bundle \(L\).
Since Fourier-Mukai pairs are stable under base change, we may assume \(k\) is algebraically closed.
First, consider case (2).  
Let \(\eta \in X\) be the generic point.  
Since \(\eta \in \Supp(L)\), we have \(\overline{\{\eta\}} \subset \Supp(L)\).  
By \cite[Lemma 5.24]{It}, it follows that \(\tSpec^{\FM}\perf(X) \subset \overline{\{\eta\}}\).
Next, consider case (4).  
By \cite[Lemma 5.6]{HO2}, we have \(\tSpec^{\FM}\perf(X) \subset \Supp(L)\).
\end{proof}

\begin{proof}[Proof of Theorem \ref{thm:main_outside_the_Fourier_Mukai_loci} (3)]
Let \(S, B\) be abelian varieties such that \(S\) is simple and \(\dim S \ge 2\), and let \(f : X \to S \times B\) be an isogeny.  
Take an affine open neighborhood \(U \subset B\) of the base point \(o_B\), and set \(Z = S \times (B \setminus U)\). Let \(S_0\) denote the set of closed points of \(S\), and let \(\Lambda\) be a well-ordered set equipped with a bijection \(\Lambda \to S_0\), \(\lambda \mapsto s_\lambda\).
Define
\[
I_\lambda = \ang{\sO_X, \perf_{f^{-1}(Z)}(X), \perf_{f^{-1}(Z_\lambda \times B)}(X)}, \quad \text{where } Z_\lambda = \{ s_{\lambda'} \mid \lambda' < \lambda \}.
\]
We first claim that \(I_\lambda \subsetneq I_{\lambda+1}\) for all \(\lambda \in \Lambda\).
Fix a homomorphism \(g: S \to X\) of abelian varieties such that the composition \(p_S \circ f \circ g : S \to S\)
is an isogeny and \(p_B \circ g \circ f : X \to B\) is zero.  
Since
\[
g^* I_\lambda = \ang{\sO_S, \perf_{(f\circ g)^{-1}(Z_\lambda \times B)}(S)},
\]
the claim follows from Lemma \ref{lem:increasing_sequence} applied to \(S\).

Next, we claim that \(\Supp(I_1)\) contains the Fourier-Mukai locus.  
Let \(Y\) be a Fourier-Mukai partner of \(X\) with a derived equivalence \(\phi : \perf(Y) \to \perf(X)\).
We may assume \(k\) is an algebraically closed field after base change. 
We identify \(Y\) with an open subscheme of \(\tSpec \perf(X)\) via the morphism of Theorem \ref{thm:foundamel_results_for_Matsui_spectrum} (5).  
Since \(\Supp(I_1)\) is closed, it suffices to show that \(y \in \Supp(I_1)\) for all closed points \(y \in Y\).
Put \(E = \phi(k(y))\).  
The prime thick subcategory \(P_y \subset \perf(X)\) corresponding to \(y\) is given by
\[
P_y = \{ F \in \perf(X) \mid \Hom^*(F,E) = 0 \}.
\]
Thus, \(y \in \Supp(I_1)\) if and only if \(\Hom^*(F, E) \neq 0\) for at least one of \(F = \sO_X\), \(F \in \perf_{f^{-1}(Z)}(X)\), or \(F \in \perf_{f^{-1}(\{x_0\} \times B)}(X)\).
The object \(E\) is a point object in the sense of \cite{JO}.  
By \cite{JO}, there exists a abelian subvariety \(X' \subset X\) and a closed point \(x \in X\) such that \(\Supp_X(E) = x + X'\).
If the restriction of \(p_B \circ f\) to \(X'\) is nonzero, then
\(\Supp_X(E) \cap f^{-1}(Z) \neq \emptyset\), and hence there exists \(F \in \perf_{f^{-1}(Z)}(X)\) with \(\Hom^*(F,E) \neq 0\).
If the restriction of \(p_S \circ f\) to \(X'\) is nonzero, since \(S\) is simple, this restriction is surjective.
Therefore, \(\Supp_X(E) \cap f^{-1}(\{x_0\} \times B) \neq \emptyset\), and \(\Hom^*(F,E) \neq 0\) for some \(F \in \perf_{f^{-1}(\{x_0\} \times B)}(X)\).
If both restrictions vanish, then since \(f\) is an isogeny, \(X'\) is trivial and \(\Supp_X(E) = \{x\}\) is a singleton.  
In this case, \(\Hom^*(\sO_X, E) \neq 0\).
\end{proof}

\section*{Acknowledgement}

I am grateful to Seidai Yasuda, Takumi Asano, and Kango Ito for their kind support in the course of writing this paper and for their help with proofreading.

This work was supported by JST SPRING, Grant Number JPMJSP2119.

\bibliographystyle{amsalpha}
\bibliography{biblatex.bib}
\end{document}